\documentclass[12pt]{article}
\usepackage{amssymb,latexsym,amsmath,amsbsy,amsthm,amsxtra,amsgen}
\newfont{\msbm}{msbm10 at 11pt}

\newcommand{\mseen}{$M$-seen}
\newcommand{\mseeing}{$M$-seeing}

\newcommand{\NN}{\mathbb N}

\newcommand{\q}{\quad}

\def\ni{\noindent}

\def\la{\lambda}

\newtheorem{Theo}{Theorem}
\newtheorem{Lemma}[Theo]{Lemma}
\newtheorem{Cor}[Theo]{Corollary}
\newtheorem{Prop}[Theo]{Proposition}

\begin{document}
\title{Percolation of arbitrary words in one dimension}
\author{Geoffrey R. Grimmett, Thomas M. Liggett and Thomas Richthammer\\\\
University of Cambridge and University of California, Los Angeles\\}
\maketitle

\begin{abstract}
We consider a type of long-range percolation 
problem on the positive integers, motivated by earlier work 
of others on the appearance of (in)finite words within a site percolation model. 
The main issue is whether a given infinite
binary word appears within an iid Bernoulli sequence at locations that satisfy
certain constraints. We settle the issue in some cases, and provide partial results in others.

\end{abstract}



\footnote{Research supported in part by NSF Grants DMS-0301795 and DMS-0707226.}

\section{Introduction}

Let $W=(w_1,w_2,\dots )\in\{0,1\}^\NN$ be an infinite binary word, 
and $X = (X_1,X_2,\dots)$ and $Y=(Y_1, Y_2,\dots )$ be independent Bernoulli sequences (i.e., iid sequences of Bernoulli random variables) 
with parameters  $p_X=P(X_i=1),p_Y=P(Y_j=1) \in (0,1)$. 
Let $M$ be a positive integer. An \emph{admissible ($M$-)embedding} 
of $W$ in $Y$ is a sequence $(m_i:i \ge 1)$ of integers such that
$Y_{m_i}=w_i$  and $1 \le m_{i}-m_{i-1}\leq M$ for each $i\geq 1$. 
(By default, we take $m_0=0$.)  We say that
\emph{$W$ is \mseen\ in $Y$} if there exists an $M$-admissible 
embedding of $W$ in $Y$. In this paper, we ask
whether or not the events $\{W \text{ is \mseen\ in } Y\}$ 
and $\{X \text{ is \mseen\ in } Y\}$ can have strictly positive probability. 

This question
is motivated by analogous questions considered in \cite{BK,KSZ0,KSZ} concerning
percolation of words on graphs such as $Z^d$. These
questions were partially answered for large $d$ in \cite{BK}, and on a
modified version of $Z^2$ in \cite{KSZ}. A version of the above question has been answered
in the affirmative in \cite{L} for large $M$ and $d=2$.
Our problem may be set in the context of long-range percolation, through a consideration of the
oriented graph with vertices $\{1,2,\dots \}$ in which there is an edge from $i$ to $j$ if
$1\le j-i\leq M$.
In this setting, our problem
corresponds to ordinary site percolation when $W$ is the constant word $w_i\equiv 1$, 
and to so-called $AB$ site percolation when $W$ is the alternating word with
$w_{2i}\equiv 0, w_{2i-1}\equiv 1$ (or \emph{vice versa}).

A further formulation of the problem resembles the famous problem of the clairvoyant
demon posed by Peter Winkler. As above, let $X$ and $Y$ be independent Bernoulli sequences with parameters $p_X,p_Y$; for simplicity we assume $p_X = p_Y = \frac12$.
We color the point $(i,j)$ of the first orthant of the square lattice $Z^2$
\emph{red} if $X_i=Y_j$.  Let the origin $(0,0)$ be red also. For $M \ge 1$, we define an
$M$-admissible path to be an infinite sequence $m=(m_0,m_1,m_2,\dots)$ satisfying 
$m_0=0$ and $1 \le m_{i+1} - m_i \le M$ for
all $i$, such that every point $(i,m_i)$, $i\ge 0$, is red. There exists an $M$-admissible
path if and only if $X$ is \mseen\ in $Y$. For references to the
clairvoyant demon problem, and for solutions to the related problem in which the
admissible paths for that problem are permitted to move upwards or downwards at each stage, see \cite{BBS, W}. 

Since $Y$ contains arbitrarily long sequences of 0's and arbitrarily long
sequences of 1's, it is easy to see that $P(W\text{ is \mseen\ in }Y)=0$ when $W$ is periodic. 
However, the situation for general words is not so clear. In Section 5, we will show that the truth of the statements  
``for every $M$, $P(X\text{ is \mseen\ in } Y)=0$" and 
``for every $M$, $P(W\text{ is \mseen\ in } Y)=0$" 
(for an arbitrary infinite word $W$) is 
independent of the parameters $p_X,p_Y$ of the Bernoulli sequences 
$X,Y$; see Theorem \ref{nondependence}. Therefore, except in that section, 
we will assume that 
$$
p_X = p_Y =\tfrac 12.
$$ 

In order to gain some insight into our problem, we consider the probability of 
\mseeing\ finite words $W$, and particularly how this probability depends on $W$.
Let 
$$
\alpha=1-2^{-M},\q \beta=2^{-M}.
$$
It is easy to check (as we will do in Section 3) that the probability 
of \mseeing\ a given word of length $n$ is minimized by the constant 
word $W=(1,1,\dots ,1)$ of length $n$, and that in this case, 
this probability equals $\alpha^n$. In the other direction, we consider the alternating word
 $A_n = (1,0,1,0,\dots)$ of length $n$; we could equally consider
the alternating word beginning with 0.
Since the infinite alternating word is periodic, the probability 
\begin{equation}
v_n = P(A_n \text{ is \mseen\ in } Y)
\label{vndef}
\end{equation}
tends to zero as $n\rightarrow\infty$. In Section 2, we will show how to compute $v_n$ exactly, and hence determine the exponential  rate at which this probability tends to zero. If $M=2$, for
example, $v_n\sim c(0.85\dots )^n$.  We will prove that the alternating word is
most likely to be seen in two cases: 

\begin{Theo} {\rm(a)} Let $M=2$ and $n \ge 1$. For any word $W$ of length $n$,
$$
P(W\text{\rm\ is \mseen\ in }Y)\leq v_n.
$$

\noindent
{\rm(b)} Let $M\ge 2$, and let $W_{p,q}$ be the word $(1,1,1,\dots ,0,0,0)$ comprising 
$p$ $1$'s followed by $q$ $0$'s. Then 
$$
P(W_{p,q}\text{\rm\ is \mseen\ in }Y)\leq v_{p+q}.
$$
\end{Theo}

The first part of this theorem will be proved in Section 2, and the second in Section 4. 
As a consequence of Theorem 1(a), we have the following solution to our
main problem in case $M=2$:

\begin{Cor} If $M=2$, $P(W\text{\rm\ is \mseen\ in }Y)=0$ for every
infinite word $W$, and thus $P(X\text{\rm\ is \mseen\ in }Y)=0$. 
\end{Cor}

A useful tool in our analysis is the following sequence of `spacing' random variables.
Given a finite or infinite word $W$, define
$T_0 = 0$ and, recursively, 
$$
T_{k+1}=\min\{i>T_k: Y_i=w_{k+1}\},\quad \tau_{k+1}= T_{k+1} - T_k.
$$
Note that, while the values of $\tau_1,\tau_2,\dots $  depend on the choice of $W$, 
for any $W$ they are
iid random variables with the geometric distribution with parameter $\frac 12$.
The values of $\tau_1,\tau_2,\dots ,\tau_n$ do not in general determine whether 
or not the word $(w_1,w_2,\dots , w_n)$ is \mseen. (An example illustrating this is given
in Section 3.)
However, they do so for the constant and alternating words.

\begin{Theo}\label{spacing} {\rm(a)} The constant word of length $n$ is \mseen\ in $Y$ if and only if $\tau_k\leq M$
for all $1\leq k\leq n$.

\noindent
{\rm(b)} The alternating word $A_n$ of length $n$ is \mseen\ in $Y$ 
if and only if 
$$
T_k\leq kM\text{ for all }1\leq k\leq n\quad 
\text{and}\quad T_k-T_j<(k-j+1)M\text{ for all }0\leq j<k\leq
n.
$$ 
\end{Theo}

This theorem will be proved in Section 3.

Let $W$ be a word of length $n$, and let $N_n=N_n(W)$ be the number of
$M$-admissible embeddings of $W$ in $Y$.
It is easy to see that 
$$
E(N_n) = (M/2)^n.
$$
The second moment of $N_n$ can be expressed in the following way.
Let $J=(J_0.J_1,J_2,\dots)$ and $K=(K_0,K_1,K_2,\dots)$  be independent
random walks on $Z$ starting at $J_0=K_0=0$ with, as step-size distribution, the uniform distribution
on the finite set $\{1,2,\dots,M\}$.

\begin{Theo}\label{grg}
{\rm(a)}
For any word $W$ of length $n$, and any $M\ge 1$,
\begin{equation}
E(N_n^2) = E(N_n)^2 E\left(\prod_{(r,s): J_r=K_s} 2 \cdot 1(w_r=w_s)\right),
\label{Nvar}
\end{equation}
where the product is over $r,s \in \{1,2,\dots,n\}$, and $1(A)$ denotes the indicator
function of $A$. 

\noindent{\rm(b)}
Let $X$ be the random word of length $n$, comprising 
random letters with the Bernoulli $(\tfrac12)$
distribution. Then
$$
E(N_n(X)^2) = E(N_n(X))^2 E(2^{Z_n}),
$$
where $Z_n$ is the number of visits to zero between times $1$ and $n$ made by the random walk
$J-K$.
\end{Theo}

This theorem will be proved in Section 6. 
We will also see there that $E(2^{Z_n})$ is asymptotic to a constant multiple of $c_M^n$ for some $c_M>1$.
In the case $M=2$, we have $c_2=\frac43$. By part (a) of this theorem,
the constant word maximizes the variance of $N_n(W)$. 

\section{Recursions for $v_n$}

Let $A_n = (1,0,1,0,\dots)$ be the alternating word of length $n$ starting with 1.
(By symmetry,  probabilities for alternating words starting with 0 are the same as for $A_n$.)
In this section, we first compute the $v_n$ given in \eqref{vndef},  and 
then we prove Theorem 1(a). The computation of $v_n$ is done recursively; the
recursions satisfied by $v_n$ will be used in Section 4 in the proof of Theorem 1(b).

If an admissible  embedding of $A_n$ in $Y$  exists, 
we define the \emph{standard embedding} to be that whose sequence of positions 
$(m_i: i \ge 1)$ is earliest in the usual lexicographic order.
We will use the following notation:
\begin{itemize}
\item $v_{n,k} :=$ the probability that $A_n$ possesses an admissible 
embedding, and its standard embedding starts at position $k$. 
\item $v_n := \sum_{k=1}^M v_{n,k}=$ the probability that $A_n$
possesses an admissible embedding.
\item $v_n' := v_{n,M}.$
\end{itemize}

\begin{Prop} The sequences $v_n$ and $v_n'$ satisfy the following recursions:
\begin{equation}
v_n=\alpha v_{n-1}+(\alpha-M\beta)v_{n-1}'\quad\text{and} 
\quad v_n'=\beta v_{n-1}+(M-1)\beta v_{n-1}'
\label{pairrecursion}
\end{equation}
for $n\geq 1$, with initial conditions $v_0=1$ and $v_0'=0$, and
\begin{equation}
v_{n+1}=(\alpha+(M-1)\beta)v_n-\beta(M-2\alpha)v_{n-1}
\label{singlerecursion}
\end{equation}
for $n\geq 1$, with initial conditions $v_0=1$ and $v_1=\alpha$.
\end{Prop}

Note that in \eqref{pairrecursion},
unlike \eqref{singlerecursion},  all the coefficients 
are nonnegative. This will enable us to compare solutions to
recursive inequalities.

It is easy to solve the recursion \eqref{singlerecursion} explicitly. 
The characteristic polynomial is 
\[
f(\la) = \la^2 -  (\alpha+(M-1)\beta) \la + \beta(M-2\alpha).
\]
Note that, for $M \ge 2$,
\begin{align*}
f(0) =\beta(M-2 \alpha) >0,\q &f(M\beta)=2\beta(M\beta-\alpha)< 0,\\
f(\alpha) =\beta(M\beta-\alpha)<0,\q 
&f(1)= 2\beta^2 >0.
\end{align*}
Therefore, the two roots of $f$ lie in the disjoint intervals $(0,M\beta)$ and $(\alpha,1)$, respectively.
This implies that $v_n \to 0$ exponentially fast as $n\to\infty$, and is
a quantification of the observation that $Y$ contains
no admissible embedding of $A_n$ if there appears, sufficiently early
in $Y$, a consecutive subsequence of $M+1$ letters all of which are $0$ (respectively $1$).
A word of caution: while $v_n$ tends to zero exponentially rapidly, the convergence
can still in a sense be quite slow. For example, if $M=5$, the larger root of
$f$ is $.9978\dots$.

\begin{proof}
A Bernoulli sequence $Y$ contributing to $v_{n,k}$ must satisfy $Y_k=1$, 
and the preceding sequence $(Y_1,\dots ,Y_{k-1})$ cannot contain 1 and  0 in that order, since
if it did, there would be an admissible embedding starting before position $k$.
Therefore, $(Y_1,\dots ,Y_{k-1})$ must be of the form
$(0,\dots,0,1,\dots,1)$. We distinguish two cases for the starting 
sequence: all 0's (Case 1) and at least one 1 (Case 2). 
In Case 1, the only condition on $(Y_{k+1},\dots )$ is that it must contain an admissible embedding
of the remainder of $A_n$. 
In Case 2, $(Y_{k+1},\dots )$ must contain an admissible embedding of the remainder
whose standard embedding starts at (relative) position $M$. 
(If it started earlier, then the 1 at position $k-1$ would 
initiate an earlier embedding.) This yields the recursion
\begin{equation}
v_{n,k} = \frac{1}{2^k} v_{n-1} + \frac{k-1}{2^k} v_{n-1,M}
\q\text{for } 1 \le k \le M,\ n \ge 1, 
\label{recur2}
\end{equation}
with initial condition
$v_{1,k} = 2^{-k}$ for $1 \le k \le M$.

The first relation in  \eqref{pairrecursion} is obtained by summing 
\eqref{recur2} over $k$, and the second by setting $k=M$.
Finally, one can eliminate $v_n'$ from \eqref{pairrecursion}
to obtain \eqref{singlerecursion}. 
\end{proof}

\begin{proof}[Proof of Theorem 1(a)]
Let $M=2$  and let
$W \in \{0,1\}^n$ be a word of length $n$.
For $1 \le m \le n$, let $W_m$ be the word comprising 
the last $m$ digits of $W$. 
We will use notation similar to that at the beginning of this section.
\begin{itemize}
\item $w_{m,k} :=$ the probability that $W_m$ possesses an admissible embedding, and its
standard embedding starts at position 
$k \in \{1,2\}$. 
\item $w_m :=  w_{m,1} + w_{m,2} = $ 
the probability that $W_m$ has an admissible embedding.
\item 
$w_m' := w_{m,2}$
\end{itemize}
We follow the same procedure as we did for the alternating word $A_n$. Unlike
that case, we shall obtain only a recursive estimate 
from above.

Denote the first digit of $W_m$ by $a \in \{0,1\}$, 
and let $b=1-a$ be the complementary digit. 
A Bernoulli sequence $Y$ contributing to $w_{m,1}$ must contain 
$a$ at position 1 and the word 
$W_{m-1}$ following. 
This gives 
\begin{equation}
w_{m,1} = \tfrac 1 2 w_{m-1}.
\label{2.1}
\end{equation} 
If $Y$ contributes to $w_{m,2}$,
there are two cases. 
If $Y$ starts with $a$, 
then the second digit must also be $a$, and subsequently
the $Y$ must contain $W_{m-1}$, but not 
starting at the next digit. (Otherwise there 
would be a standard embedding starting at position 1.) 
If $Y$ starts with $b$, 
then the second digit has to be $a$, and subsequently
$Y$ must contain $W_{m-1}$.
This gives 
\begin{equation}
w_{m,2} \le \tfrac 1 4 w_{m-1,2} + \tfrac 1 4 w_{m-1}.
\label{2.2}
\end{equation}

Note that equality need not hold in \eqref{2.2}.
Suppose that both $Y$ and $W_m$ begin with the letters $aa$, and that $(Y_3,Y_4,\dots)$
contains an admissible embedding of both $W_{m-2}$ and $W_{m-1}$. In this
case, $Y$ does not contribute to the left side of \eqref{2.2} but it does to the first
term on the right side.

From \eqref{2.1}--\eqref{2.2}, we deduce the recursive inequalities
\begin{equation*}
w_{m} \le \tfrac 3 4 w_{m-1}+ \tfrac 1 4 w_{m-1}' \; 
 \text{ and } \;  
w_{m}' \le\tfrac 1 4 w_{m-1}+ \tfrac 1 4 w_{m-1}', \quad 
\text{ where }  \quad w_1 = \tfrac 3 4, w_1'= \tfrac 1 4. 
\end{equation*}
By comparison with the recursion formula \eqref{pairrecursion}
with $M=2$,
\begin{equation} 
v_{m} = \tfrac 3 4 v_{m-1} + \tfrac 1 4  v_{m-1}'
\; \text { and } \; 
v_m' = \tfrac 1 4 v_{m-1} + \tfrac1 4 v_{m-1}', \quad 
\text{where} \quad v_{1} = \tfrac 3 4,\ v_1' = \tfrac{1}{4},
\label{recurm=2}
\end{equation}
we obtain by induction on $m$ and the positivity of the coefficients
in \eqref{recurm=2} that $w_m \le v_m$ for $1 \le m \le n$. In particular, $w_n \le v_n$
as claimed.
\end{proof}

\ni{\bf Remark.} The above method does not work for $M\ge 3$, since in
this case the 
coefficients of the recursive inequalities for $w_n$, $w_n'$ 
do not match the coefficients of the recursion 
for $v_n$, $v_n'$. 

\section{Relations to the spacing random variables}

In this section, we prove several results relating \mseen\ finite words to
inequalities satisfied by the spacing variables $\tau_k$ and their partial sums.
When the word is either constant or alternating, these are equivalences, and
were stated as Theorem  \ref{spacing} in the Introduction. For general words, we only have
one direction --- if the word is seen, then the spacing variables satisfy certain
inequalities. We will say that $W=(w_1, w_2,\dots )$ is seen at $(m_1,m_2,\dots )$ 
if $m_1<m_2<\cdots$ and $Y_{m_i}=w_i$ for each $i$. 

First, we give an example to show that seeing a word $W$ of length $n$ is not in general
determined by the values of $\tau_1,\dots ,\tau_n$.  Suppose $n=4$, $M=2$, $W=(1,1,0,0)$,
and the Bernoulli sequence starts with $110110\cdots$. Then $\tau_1=1$, $\tau_2=1$, $\tau_3=1$,
$\tau_4=3.$ If $W$ is to be seen, then it must be seen at locations $m_1=2$, $m_2=4$, $m_3=6$,
and $m_4=7$ or $8$. Thus, it is seen if and only if one of the next two digits in the Bernoulli
sequence is a 0, but this cannot be determined from the first four $\tau_k$.

\begin{Prop}\label{P:firstbound} Let $W=(w_1,\dots , w_n)$. 
If $W$ is seen at $(m_1,\dots ,m_n)$, then
$T_k\leq m_k$ for $1\leq k\leq n$. In particular, if $W$ is \mseen, 
then $T_k\leq kM$ for all $1\leq k\leq n.$
\end{Prop}

\begin{proof} 
Let $m_0 := T_0 = 0$. We will prove $T_k \le m_k$ by induction on $k$. 
For the induction step we assume $T_k \le m_k$. Let $a = w_{k+1}$. 
By definition $T_{k+1}$ is the first location of an $a$ after location $T_{k}$, and $m_{k+1}$ 
is some location of an $a$ after location $m_{k} \ge T_k$, which 
immediately implies $T_{k+1} \le m_{k+1}$. 
Finally, if $W$ is \mseen, then 
$$
T_k\leq m_k=\sum_{i=1}^k(m_i-m_{i-1})\leq kM.
$$
\end{proof}

The next result implies that the probability of \mseeing\ a word
of length $n$ is minimized by the constant word, and in that case, this probability
is $\alpha^n$.

\begin{Prop}\label{P:constant} 
{\rm(a)} If $W$ is a word of length $n$ and $\tau_1\leq M$, $\dots$,
$\tau_n\leq M$, then $W$ is \mseen. 

\ni{\rm(b)} If a constant word of length $n$ is \mseen, then $\tau_1\leq M$, $\dots$, $\tau_n\leq M$.
\end{Prop}

\begin{proof} For part (a), note that if $\tau_i\leq M$ for each $i\leq n$, 
then $W$ is \mseen\ at $(T_1,\dots ,T_n)$. For part (b), suppose that 
the constant word $(a,a,\dots ,a)$ of length $n$ is seen at
$(m_1,\dots ,m_n)$  where $1 \le m_i-m_{i-1}\leq M$ for each $i$, i.e. 
up to location $m_n$ there is no block of $M$ consecutive non $a$'s. 
As $T_n \le m_n$ by Proposition \ref{P:firstbound}, this 
implies $\tau_i\leq M$ for all $i\leq n$. 
\end{proof}

\begin{Prop}\label{P:secondbound} Let $W=(w_1,\dots , w_n)$. 
Suppose that $W$ is seen at $(m_1,\dots ,m_n)$ and
that $0<m_{i+1}-m_i\leq M$ for each $i$. If $w_k\neq w_{k+1}=\cdots =w_l$ for  some $k+1 \le l$ 
and if $\tau_l>M$, then $T_l\leq m_{k+1}$.
\end{Prop}

\begin{proof} 
Let $a=w_{l}$. In between locations 
$T_{l-1}$ and $T_l$ there is a block of at least $M$ consecutive non $a$'s. As there have to be $a$'s at locations 
$m_{k+1} < \cdots < m_{l}$ and we have $m_{i+1}-m_i\leq M$,
this block has to be before location $m_{k+1}$
or after location $m_l$, i.e. $T_l \le m_{k+1}$ or 
$m_l \le T_{l-1}$. 
By Proposition \ref{P:firstbound} we have $T_l \le m_l$, so the second alternative is not possible.
\end{proof} 

\begin{Prop}\label{P:alternating}  Let $A_n$ be an alternating word of length $n$. Then
$A_n$ is \mseen\ if and only if 
\begin{equation}
T_k\leq kM\text{ for all } 1\leq k\leq n\quad\text{ and }
\quad T_k-T_j<(k-j+1)M\text{ for all }0\leq j<k\leq n.
\label{alternating1}
\end{equation}
\end{Prop}

\begin{proof} A special property of an alternating word is that 
\begin{equation}
Y_i=w_k \text{ for } T_k\leq i<T_{k+1}.\label{special}
\end{equation}
Define
$$
S_k=\min\{T_{k+1}-1,S_{k-1}+M\},\quad S_0=0,\quad\sigma_k=S_k-S_{k-1}.
$$
Note that $S_k<T_{k+1}$ and $\sigma_k\leq M$. Since $T_k$ is strictly increasing in $k$, we see inductively
that $S_k$ is strictly increasing in $k$ also. Therefore $\sigma_k\geq 1$ for all $k\geq 1.$
Consider the statement
\begin{equation}T_k\leq S_k\text{ for all }1\leq k\leq n.\label{alternating2}\end{equation}
We will prove the following implications:
$$
\eqref{alternating2}\Rightarrow A_n \text{ is \mseen }\Rightarrow\eqref{alternating1}
\Rightarrow\eqref{alternating2}.
$$

First suppose that \eqref{alternating2} holds. Then $T_k\leq S_k<T_{k+1}$,
so that $Y_{S_k}=w_k$ by \eqref{special}. Since $\sigma_k\leq M$, it follows that $A_n$ is \mseen,
since it is seen at $(S_1,\dots ,S_n)$.

Next assume that $A_n$ is seen at $(m_1,\dots ,m_n)$ where $m_i-m_{i-1}\leq M$ for all $i$. The first
part of \eqref{alternating1} follows from Proposition \ref{P:firstbound}. To prove the second part, let $0\leq j<k\leq n$. 
As $(Y_i: T_j\leq i<T_k)$ consists of $k-j$ constant blocks, 
the interval $[T_j,T_k)$ contains at most $k-j$  consecutive 
elements from $m_1,\dots , m_n$, i.e. 
$m_l < T_j \le m_{l+1} \le m_{l+r} < T_k \le m_{l+r+1}$ 
for some $l$ and $r \le k-j$. 
So $T_k - T_j < m_{l+r+1} - m_1 = \sum_{i=l}^{l+r} (m_{i+1} - m_i) 
\leq (r+1)M$.

Finally, assume that \eqref{alternating1} holds. To prove \eqref{alternating2}, we will prove the statement 
\begin{equation}
T_k\leq S_i+(k-i)M\label{inequalityTk1}
\end{equation}
by induction on $i$ (for fixed $k$). When $i=0$, \eqref{inequalityTk1} becomes $T_k\leq kM$, which is part of
assumption \eqref{alternating1}. When $i=k$, \eqref{inequalityTk1} is $T_k\leq S_k$, which is the desired conclusion in \eqref{alternating2}.
For the induction step, suppose \eqref{inequalityTk1} holds for $i$ with $0\leq i<k$. To prove it for $i+1$, we need
to check that 
\begin{equation}
T_k\leq \min\{T_{i+2}-1,S_i+M\}+(k-i-1)M.\label{inequalityTk2}
\end{equation}
The fact that $T_k\leq T_{i+2}-1+(k-i-1)M$ follows from \eqref{alternating1}, while $T_k\leq S_i+(k-i)M$
is just the induction hypothesis. This proves \eqref{inequalityTk2}.
\end{proof}

\section{Two-block words}
We prove Theorem 1(b) in this section. For $p,q,j\geq 0$,
define
\begin{align*}
&\sigma_{p,j}=P(\tau_1\leq M,\dots ,\tau_p\leq M,\, T_{p+j}>pM), \\
&\sigma_{p,j}'=P(\tau_1\leq M,\dots ,\tau_p\leq M,\, T_{p+j} \le pM)
\end{align*}
and 
$$
u_{p,q}= \alpha^{p+q}+\beta\sum_{j=1}^{q} \alpha^{q-j} \sigma_{p,j}'
= \alpha^p-\beta\sum_{j=1}^q\alpha^{q-j}\sigma_{p,j}.
$$
Here we have used $\sigma_{p,j} + \sigma_{p,j}' = \alpha^p$.
Note that $u_{p,0}=\alpha^p$ and 
$u_{0,q}=\alpha^q$.
The next result will allow us to compute $\sigma_{p,j}$ fairly explicitly.

\begin{Lemma}\label{lemma0} For $p,j \ge 0$ and arbitrary $l$,
$$P(\tau_1\leq M,\dots ,\tau_p\leq M,\, T_{p+j}> lM)=\sum_{i=0}^p\binom pi(-\beta)^iP(T_{p+j}> (l-i)M).$$
\end{Lemma}

\begin{proof} Use the fact that for any geometric random variable $\tau$, the conditional
distribution of $\tau-M$ given $\tau>M$ is the same as the distribution of $\tau$, to write
\begin{align*} 
&P(\tau_1\leq M,\dots ,\tau_m\leq M,\, T_k> lM)-P(\tau_1\leq M,\dots ,\tau_{m+1}\leq M,\, T_k>lM)
\\
&\hskip2cm =P(\tau_1\leq M,\dots ,\tau_m\leq M,\tau_{m+1}>M,\, T_k> lM)\\
&\hskip2cm =
\beta P(\tau_1\leq M,\dots ,\tau_m\leq M,\, T_k > (l-1)M)
\end{align*}
for any $k > m$.
Now use induction on $m$, together with the relation $\binom mi+\binom m{i-1}=\binom
{m+1}i$. 
\end{proof}

\begin{Lemma} \label{lemma1} 
For $p,q\geq 0$, $P(W_{p,q}\text{\rm\ is \mseen})\leq u_{p,q}$.
\end{Lemma}

\begin{proof}
If $W_{p,q}$ is seen at  $(m_1,\dots ,m_{p+q})$ where $m_{i+1}-m_i\leq M$
for all $i$, then
$\tau_1\leq M,\dots , \tau_p\leq M$ by Proposition \ref{P:constant}(b). 
Furthermore, if $1 \le j \leq q$ and $\tau_{p+j} >M$, 
then $T_{p+j}\leq m_{p+1} \le (p+1)M$ by Proposition \ref{P:secondbound}. 
Considering the largest $j \ge 1$ (if any) for which $\tau_{p+j}>M$, we see that
$$
P(W_{p,q}\text{ is seen})
\leq 
P(\tau_1\leq M,\dots ,\tau_{p+q}\leq M) 
+\sum_{j=1}^{q} \tilde{\sigma}_{p,j}, 
$$
where 
$$ 
\tilde{\sigma}_{p,j} = P(\tau_1\leq M,\dots ,\tau_p\leq M,\tau_{p+j}>M,\tau_{p+j+1}\leq M,\dots ,\tau_{p+q}\leq M, \,T_{p+j}\leq (p+1)M).
$$
Using the same trick as in the proof of the previous lemma  we obtain  
$$ 
\tilde{\sigma}_{p,j} = 
\alpha^{q-j}P(\tau_1\leq M,\dots ,\tau_p\leq M,\tau_{p+j}>M,\,T_{p+j}\leq (p+1)M)
= \alpha^{q-j} \beta \sigma'_{p,j},
$$
so the result follows by definition of $u_{p,q}$.
\end{proof}

For any function $f_{p,q}$, $p,q\geq 0$ , define a generalized mixed second derivative by
$$
\Delta f_{p,q}=f_{p+1,q+1}-M\beta f_{p,q+1}-(\alpha-\beta)f_{p+1,q}+\beta(M-2\alpha)f_{p,q}.
$$
This particular choice of coefficients is designed to correspond to the coefficients in
\eqref{singlerecursion}. In order to do so, the middle coefficients would have to
sum to $-(\alpha+(M-1)\beta)$. This particular decomposition was chosen by computing
numerically $\Delta u_{p,q}$ for various values of the parameters, 
and checking to see which one made this expression $\leq 0$.

\begin{Lemma}\label{lemma2}  $\Delta u_{p,q}\leq 0$ for $p,q\geq 0$.\end{Lemma}
\begin{proof} If $f_{p,q}=\alpha^p$, then
$$
\Delta f_{p,q}=\alpha^{p+1}-M\beta\alpha^p-(\alpha-\beta)\alpha^{p+1}+\beta(M-2\alpha)\alpha^p=0.
$$
Therefore, if we let
$$
w_{p,q}=\sum_{j=1}^q\alpha^{q-j}\sigma_{p,j},
$$
we have $\Delta u_{p,q}=-\beta\Delta w_{p,q}$, so we need to show that $\Delta w_{p,q}\geq 0.$
We will do so by computing the generating function of this expression as a function of $q$.

By Lemma \ref{lemma0},
$$
\sigma_{p,j}=\sum_{i=0}^p\binom pi(-\beta)^{p-i}P(T_{p+j}>iM).
$$
Since $T_{m}$ is a sum of $m$ independent geometric random variables 
and can thus be interpreted as the waiting time for the $m$-th success 
in a sequence of Bernoulli experiments, and since $S_k$, the number of successes in the first $k$ of these experiments is binomially distributed, 
we have 
$$
P(T_j>k)=P(S_k<j)=\frac 1{2^k}\sum_{l=0}^{j-1}\binom kl.
$$
Therefore, since $2^{-M}=\beta$,
$$\sigma_{p,j}=\sum_{i=0}^p\binom pi(-\beta)^{p-i}\beta^i\sum_{l=0}^{p+j-1}\binom{iM}l
=\beta^p\sum_{i=0}^p\sum_{l=0}^{p+j-1}(-1)^{p-i}\binom pi\binom{iM}l.$$
It follows that for $0<x<1$,
\begin{equation}
(1-x)\sum_{j=1}^{\infty}\sigma_{p,j}x^{j-1}=\beta^p\sum_{i=0}^p\sum_{l=0}^{\infty}
\binom pi\binom{iM}l(-1)^{p-i}x^{(l-p)^+}.\label{generating}
\end{equation}
Note that for every polynomial function $g$ of degree $\deg g < p$ 
we have 
\begin{equation} 
\sum_{i=0}^p\binom pi  g(i) (-1)^{p-i} =0. \label{polynomial}
\end{equation}
It suffices to check this for all polynomial functions of the 
form $g(i) = \binom i l$, $0\leq l<p$: 
$$
\sum_{i=0}^p\binom pi  \binom i l (-1)^{p-i} = 
\binom p l \sum_{i=l}^p \binom {p-l}{p-i} (-1)^{p-i} = \binom p l (1-1)^{p-l} = 0.
$$
Applying \eqref{polynomial} to the functions $g(i) = \binom{iM}{l}$,
$0\leq l<p$, it follows that the positive part at the end of  \eqref{generating} is not needed:
$$\begin{aligned} 
(1-x)\sum_{j=1}^{\infty}\sigma_{p,j}x^{j-1}&=\beta^p\sum_{i=0}^p\sum_{l=0}^{\infty}
\binom pi\binom{iM}l(-1)^{p-i}x^{l-p}\\
&=\beta^px^{-p}\sum_{i=0}^p\binom pi(-1)^{p-i}(1+x)^{iM}\\
&=\beta^px^{-p}[(1+x)^M-1]^p.
\end{aligned}$$
Therefore,
$$(1-x)\sum_{q=1}^{\infty}w_{p,q}x^{q-1}=\frac{1-x}{1-\alpha x}\sum_{j=1}^{\infty}\sigma_{p,j}x^{j-1}
=\frac 1{1-\alpha x}\beta^px^{-p}[(1+x)^M-1]^p.$$
Using this expression, we can write
$$(1-x)\sum_{q=0}^{\infty}\Delta w_{p,q}x^{q-1}=\frac{\beta^{p+1}x^{-p}}{x^2(1-\alpha x)}
[(1+x)^M-1]^pP(x),$$
where
$$P(x)=(1+x)^M[1-(\alpha-\beta)x]-1-(M+\beta-\alpha)x+x^2(M-2\alpha).$$
Note that $P(1)=2^M(1-\alpha+\beta)-2=0$, so we may define a polynomial $Q$ by
$P(x)=(1-x)Q(x)$. If $Q$ has nonnegative coefficients, it will follow that $\Delta w_{p,q}\geq 0$
for all $p,q\geq 0$ as required. The coefficients of $Q$ are the partial sums of the coefficients
of $P$. The constant and linear terms in $P$ vanish. 
The coefficient of $x^2$ is $\binom M2-2(\alpha-M\beta),$
while for $k\geq 3$, the coefficient of $x^k$ is $\binom Mk-(\alpha-\beta)\binom M{k-1}.$
Therefore, we need to check that the following expression is nonnegative for $l\geq 2$:
\begin{align*}
&\binom M2-2(\alpha-M\beta)+\sum_{k=3}^l\bigg[\binom Mk-(\alpha-\beta)\binom M{k-1}\bigg]\\
&\hskip2cm =\binom M l+2\beta\sum_{k=2}^{l-1}\binom Mk-2(\alpha-M\beta)=\binom Ml-2\beta\sum_{k=l}^M\binom Mk.
\end{align*}
This is nonnegative for $2\leq l<M$ since $\beta\sum_{k=0}^M\binom Mk=1,$ and for $l=M$ since
then the right side above is $1-2\beta$.
\end{proof}

\begin{proof}[Proof of Theorem 1(b)] Let
$$\delta_{p,q}=v_{p+q}-u_{p,q},\quad p,q\geq 0.$$
By \eqref{pairrecursion} we have 
$v_{n+1}\geq \alpha v_n$ and thus $v_n\geq\alpha^n$, so 
\begin{equation}
\delta_{p,0} = v_p-\alpha^p \geq 0 \q \text{ and } \q 
\delta_{0,q} = v_q-\alpha^q \geq 0 \q \text{ for all } p,q \ge 0.
\label{deltanull}
\end{equation}
By \eqref{singlerecursion} and Lemma \ref{lemma2} we have 
\begin{align*}
&\delta_{p+1,q+1}-M\beta \delta_{p,q+1}-(\alpha-\beta)\delta_{p+1,q}+\beta(M-2\alpha)\delta_{p,q}\nonumber\\
&\hskip1cm=v_{p+q+2}-(\alpha+(M-1)\beta)v_{p+q+1}+\beta(M-2\alpha)v_{p+q}-\Delta u_{p,q}\geq 0,
\end{align*} 
which can be rewritten as 
\begin{equation}
\delta_{p+1,q+1}-M\beta\delta_{p,q+1} \geq(\alpha-\beta)\delta_{p+1,q}-\beta(M-2\alpha)\delta_{p,q}.
\label{3.6}
\end{equation}
We will now prove by induction on $q$ the statement that $\delta_{p+1,q}\geq M\beta\delta_{p,q}$
for all $p\geq 0$.
By \eqref{deltanull}, for the basis step 
we have to show that
$$
v_{n+1}-\alpha^{n+1}\geq M\beta(v_n-\alpha^n).
$$
This follows from 
$$v_{n+1}-M\beta v_n=(\alpha-M\beta)(v_n+v_n')\geq (\alpha-M\beta)\alpha^n,$$
where we have used \eqref{pairrecursion}, 
$\alpha-M\beta >0$, $v_n' \ge 0$, and $v_n\geq\alpha^n$. 
For the induction step, assume that the statement is
true for a given $q\geq 0$. Using \eqref{deltanull}
it follows that
$\delta_{p,q}\geq 0$ for that $q$ and all $p$. Therefore, since $0\leq M-2\alpha\leq M(\alpha-\beta)$
(which is equivalent to $M\beta\leq\alpha$), \eqref{3.6} can be written as
$$\delta_{p+1,q+1}-M\beta\delta_{p,q+1}\geq\frac{M-2\alpha}M[\delta_{p+1,q}-M\beta\delta_{p,q}]\geq 0,$$
where the final inequality follows from the induction hypothesis. This proves that
$\delta_{p+1,q}\geq M\beta\delta_{p,q}$ for all $p,q\geq 0$, and hence that $\delta_{p,q}\geq 0$
for all $p,q\geq 0$ and therefore $u_{p,q}\leq v_{p+q}$. Now apply Lemma \ref{lemma1} to complete the proof.
\end{proof}

\section{Independence of parameter choice}

Let $X = (X_n: n \ge 1)$ and $Y = (Y_n : n \ge 1)$ 
be two independent Bernoulli sequences with parameters 
$p_X$ and $p_Y$ respectively, and let $W \in \{0,1\}^\NN$ 
be an arbitrary infinite word. 

\begin{Theo}\label{nondependence}
{\rm(a)} The validity of
the assertion \lq\lq$\ \forall M \ge 2: P(X \text{ is \mseen\ in } Y) = 0$" 
does not depend on the values of $p_X,p_Y \in (0,1)$.\\
{\rm(b)} The validity of
the assertion \lq\lq$\ \forall M \ge 2: P(W \text{ is \mseen\ in } Y) = 0$" 
does not depend on the value of $p_Y \in (0,1)$.
\end{Theo}

The main idea of the proof of Theorem \ref{nondependence} 
is to use a coupling of two Bernoulli sequences so that one can 
be \mseen\ in the other for sufficiently large $M$. 

\begin{Lemma}\label{lemmacoupling}  
Let $X,X'$ be Bernoulli sequences with parameters $p,p' \in (0,1)$ respectively.\\
{\rm(a)} If $p' \in [p^2,1-(1-p)^2]$, there is a coupling  such that $X'$ can be $3$-seen in $X$.\\
{\rm(b)}
There is an $M \ge 2$ and a coupling such that $X'$ can be \mseen\ in $X$. 
\end{Lemma}

\begin{proof} {\rm(a)} Let $X$ be a Bernoulli sequence with parameter $p$.
For a given sequence $x \in \{0,1\}^{\mathbb N}$,
we define a random subsequence $F(x)$ by 
partitioning $x$ into disjoint blocks each comprising 2 consecutive letters, 
and replacing every block  $00$ by $0$, $11$ by $1$, and 
$01$ and $10$ by $1$ with probability $p_1$ 
and by $0$ with probability $p_0 = 1-p_1$. This is done 
independently for each block and independently of 
the choice of the Bernoulli sequence $X$. 
The sequence $X' := F(X)$  is a Bernoulli
sequence with parameter
$$
p' = p^2  + 2p(1-p)p_1
$$
such that $X'$ can be 3-seen in $X$ (since for
every $k$, $X'_k=X_{2k}$ or $X_{2k-1}$).
Since $p_1$ can be chosen arbitrarily in $[0,1]$, 
$p'$ has a possible range of $[p^2,1-(1-p)^2]$. \\
{\rm(b)} 
Let $f_1(p) = p^2$ and $f_2(p)= 1-(1-p)^2$.
We note that $f_1(p) \le p \le f_2(p)$ and 
$f_1^k(p) \to 0$  and $f_2^k(p) \to 1$ for $k \to \infty$, 
where $f^k(p) = f \circ \cdots \circ f (p)$. Thus for arbitrary 
given $p,p' \in (0,1)$, there exist $p_0,\ldots,p_k \in (0,1)$ 
such that $p_0 = p$, $p_k = p'$ and $p_{i+1} \in [f_1(p_i),f_2(p_i)]$
for all $i$. Thus by {\rm(a)} there exist Bernoulli sequences 
$X^{(0)}, \ldots, X^{(k)}$ such that $X^{(i)}$ has parameter $p_i$ 
and $X^{(i+1)}$ can be $3$-seen in $X^{(i)}$ for all $i$. 
Thus $X' := X^{(k)}$ can be $3^k$-seen in $X := X^{(0)}$.
\end{proof}

\begin{proof}[Proof of Theorem \ref{nondependence}]
{\rm(a)} Let $p_X,p_Y,p_X',p_Y' \in (0,1)$. By Lemma 
\ref{lemmacoupling} there are Bernoulli sequences $X,Y,X',Y'$ 
with these parameters such that $X'$ can be $M_X$-seen in 
$X$, $Y$ can be $M_Y$-seen in $Y'$, and $(X,X')$ is independent 
of $(Y,Y')$. In particular  we have 
$$
\{X \text{ is \mseen\ in }Y\} \subset \{X' \text{ is $M_XM_YM$-seen in }Y'\} \q \text{ for all } M.
$$
Thus the validity of \lq\lq$\forall M' \ge 2: P(X' \text{ is $M'$-seen in } Y') = 0$"
implies the validity of \lq\lq$\forall M \ge 2: P(X \text{ is \mseen\ in } Y) = 0$". \\
{\rm(b)} Let $p_Y,p_Y' \in (0,1)$. By Lemma 
\ref{lemmacoupling} there are Bernoulli sequences $Y,Y'$ 
with these parameters such that $Y$ can be $M_Y$-seen in $Y'$. In particular  we have 
$$
\{W \text{ is \mseen\ in }Y\} \subset \{W \text{ is $M_YM$-seen in }Y'\} \q \text{ for all } M.
$$
Thus the validity of \lq\lq$\forall M' \ge 2: P(W \text{ is $M'$-seen in } Y') = 0$"
implies the validity of \lq\lq$\forall M \ge 2: P(W \text{ is \mseen\ in } Y) = 0$". \end{proof}

\section{Variance of the number of embeddings}

We prove Theorem \ref{grg}.
Let $W = (w_1,w_2,\dots,w_n)$ be a word of length $n$, and let
$j=(j_0,j_1,\dots,j_n)$ and $k=(k_0,k_1,\dots,k_n)$ be strictly increasing sequences
of integers with $j_0=k_0=0$ and gaps not exceeding $M$. Then
\begin{align*}
E(N_n^2) &= \sum_{j,k}P(\text{$j$ and $k$ are $M$-admissible embeddings})\\
&= \sum_{j,k} (\tfrac12)^{|j\cup k|}I_{j,k} = \sum_{j,k} (\tfrac12)^{2n-|j\cap k|}I_{j,k},
\end{align*}
where 
$$
I_{j,k} = I_{j,k}(W) = \prod_{(r,s)\ne 0:\, j_r=k_s} 1(w_r = w_s),
$$
$j\cup k = \{j_1,j_2\dots,j_n\}\cup
\{k_1,k_2,\dots,k_n\}$ viewed as a set, and $|j \cap k| = |\{(r,s) \ne (0,0) : j_r = k_s\}|$.
Therefore,
$$
E(N_n^2) = (M/2)^{2n} E(2^{|J\cap K|}I_{J,K}),
$$
as claimed in part (a).

Part (b) follows from the fact that, for the random word $X$,
$$
E(I_{j,k}(X)) = (\tfrac12)^{|j\cap k| - Z(j,k)},
$$
where $Z(j,k) = |\{l \ne 0: j_l=k_l\}|$.

To determine the asymptotics of $E(2^{Z_n})$, we proceed as follows. Let $\tau, \tau_1,
\tau_2,...$ be iid with the distribution of the hitting time of 0 for the random walk 
$J-K$ starting at 0. Then $P(Z_n\geq k)=P(\tau_1+\cdots+\tau_k\leq n)$, so
$$
E(2^{Z_n})=1+\sum_{k=1}^{\infty}2^{k-1}P(\tau_1+\cdots+\tau_k\leq n),
$$
and therefore, for small positive $x$,
\begin{equation}
(1-x)\sum_{n=0}^{\infty}x^nE(2^{Z_n})=\frac{1-Ex^{\tau}}{1-2Ex^{\tau}}.
\label{genftn}
\end{equation}
The function $x\rightarrow Ex^{\tau}$ is smooth on $[0,1)$, so if we define $c=c_M>1$ by
$Ec^{-\tau}=\frac12$, the right side of \eqref{genftn} is asymptotic to a constant multiple
of $(1-cx)^{-1}$ as $x\uparrow c^{-1}$. By a Tauberian theorem (e.g., Theorem 5
in Section XIII.5 of \cite{FII}), it follows that $E(2^{Z_n})$ is asymptotic to a constant
multiple of $c^n$. When $M=2$, $Ex^{\tau}=1-\sqrt{1-x}$ (see, for example,
Section XIV.4 of \cite{FI}), so $c_2=\frac43$.

In order to get full benefit from the Tauberian theorem, we need to know that the sequence $E(2^{Z_n})c^{-n}$ is
monotone; otherwise we would only have convergence in the Ces\`aro sense. 
We check this next. Let $u_n=P(J_n-K_n=0)$, $v_0=1,$
\begin{equation} v_n=\sum_{k=1}^nu_kv_{n-k},\quad n\geq 1,\label{renewal}\end{equation}
and $V_n=\sum_{k=0}^n v_k$. We will check that
\begin{equation}
u_n\downarrow,\label{down}\end{equation}
\begin{equation}V_n^2\geq V_{n+1}V_{n-1}, \quad n\geq 1
\label{logconvex}\end{equation}
and 
\begin{equation} 
E(2^{Z_n})=V_n.
\label{identity}\end{equation}
Once these are checked, it will follow from \eqref{logconvex} that $V_{n+1}/V_n$ is
decreasing, and by \eqref{identity} and the Ces\`aro convergence noted above, 
$V_{n+1}/V_n \downarrow c$. Therefore $V_{n+1}\geq cV_n$, so $V_n/c^n$ is increasing as required.

To check \eqref{down}, note first that by the Schwarz inequality,
\begin{align*}P(J_n-K_n=m)&=\sum_lP(K_n=l)P(J_n=l+m)\\&\leq\sqrt{\sum_lP^2(K_n=l)
\sum_lP^2(J_n=l+m)}\\&=\sum_lP^2(K_n=l)=u_n.\end{align*}
Then use this to write
$$u_{n+1}=\sum_mP(J_n-K_n=m)P(J_1-K_1=-m)\leq u_n.$$

We check \eqref{logconvex} by induction on $n$. Note that it is equivalent to
\begin{equation}
\frac{v_n}{V_{n-1}}\geq \frac{v_{n+1}}{V_{n}}.
\label{ratio}\end{equation}
Dividing \eqref{renewal} by $V_{n-1}$ expresses the left side of \eqref{ratio} as an
average of  $u_1,\dots,u_n$, and of course the right side is an average of
$u_1,\dots,u_{n+1}$. By \eqref{down}, it suffices to check that the two averaging measures are
stochastically ordered. But this is equivalent to $V_{n-1}V_j\geq V_nV_{j-1}$
for $1\leq j\leq n$, which is a consequence of the induction hypothesis.

Finally, we check \eqref{identity}. Write
$$
E(2^{Z_n})=E\prod_{i=1}^n\bigg(1+1(J_i=K_i)\bigg)=\sum_{A\subseteq\{1,\dots,n\}}
P(J_i=K_i,\, \forall\ i\in A).$$
So, it is enough to show that
$$
v_n=\sum_{n\in A\subseteq\{1,\dots,n\}}P(J_i=K_i,\, \forall\ i\in A),
$$
or equivalently, that the right side above satisfies the recursion \eqref{renewal}.
But this is easily checked by breaking up the sum according to the value
of the smallest element $k$ of $A$.

\section*{Acknowledgments} 
We acknowledge a helpful observation of Alexander Holroyd.
This work was done in part during visits of the first author
to the University of California, Los Angeles, and to the Isaac Newton Institute, Cambridge.

\bigskip

\ni  Statistical Laboratory

\ni University of Cambridge

\ni Wilberforce Road

\ni Cambridge CB3 0WB, UK

\bigskip

\ni email: g.r.grimmett@statslab.cam.ac.uk

\bigskip

\noindent Department of Mathematics

\noindent University of California, Los Angeles

\noindent 405 Hilgard Ave.

\noindent Los Angeles CA 90095, USA
\bigskip

\noindent email: tml@math.ucla.edu

\ni email: richthammer@math.ucla.edu


\begin{thebibliography}{99}

\bibitem{BBS}
Balister, P.\ N., Bollob\'as, B., Stacey, A.\ M., 
\emph{Dependent percolation in two dimensions},
Probab.\ Th.\ Rel. Fields \textbf{117}~(2000), 495--513.

\bibitem{BK}
Benjamini, I., Kesten, H., \emph{Percolation of arbitrary words in
$\{0,1\}^{\mathbb N}$}, Ann.\ Probab.\  \textbf{23}~(1995) 1024--1060.

\bibitem{FI}
Feller,  W., \emph{An Introduction to Probability Theory and its
Applications, Volume I, 3rd edition}, Wiley~(1968).


\bibitem{FII}
Feller,  W., \emph{An Introduction to Probability Theory and its
Applications, Volume II}, Wiley~(1966).

\bibitem{KSZ0}
Kesten, H., Sidoravicius, V., Zhang, Y., \emph{Almost all words are seen in
critical site percolation on the triangular lattice}, Elect.\ J.\  Probab.\  \textbf{3}~(1998) Paper \#10, 1--75.

\bibitem{KSZ}
Kesten, H., Sidoravicius, V., Zhang, Y., \emph{Percolation of arbitrary words on
the close-packed graph of $Z^2$}, Elect.\ J.\  Probab.\  \textbf{6}~(2001) Paper \#4, 1--27.

\bibitem{L}
Lima, B.\ N.\ B.\ de, \emph{A note about the truncation question in
percolation of words}, Bull.\ Braz.\ Math.\ Soc.\ New Ser.\ \textbf{39}~(2008),
183--189.

\bibitem{W} Winkler, P.,
\emph{Dependent percolation and colliding random walks},
Rand.\ Struct.\ Alg.\ 
\textbf{16}~(2000), 58--84.

\end{thebibliography}
\end{document}